\documentclass[12pt]{article}
\usepackage[T1]{fontenc}
\usepackage{lmodern}
\usepackage{amsmath,amssymb,amsthm,mathrsfs}
\usepackage[margin=1in]{geometry}
\usepackage{microtype}
\usepackage{needspace}
\usepackage[colorlinks=true,linkcolor=blue,citecolor=blue,urlcolor=blue]{hyperref}
\hypersetup{
 pdftitle={The second-order term for the largest r-fork-free families},
 pdfauthor={Yiyan Zhan, Mei Lu, Xiamiao Zhao}}

\newtheorem{theorem}{Theorem}[section]
\newtheorem{lemma}[theorem]{Lemma}
\newtheorem{proposition}[theorem]{Proposition}

\theoremstyle{remark}

\newcommand{\La}{\operatorname{La}}
\newcommand{\bm}{\boldsymbol m}
\newcommand{\cU}{\mathcal U}
\newcommand{\cV}{\mathcal V}

\title{The second-order term for the largest $r$-fork-free families}
\author{
 Yiyan Zhan\thanks{Email: \texttt{zhanyy24@mails.tsinghua.edu.cn}.}
 \quad Mei Lu\thanks{Email: \texttt{lumei@tsinghua.edu.cn}.}
 \quad Xiamiao Zhao\thanks{Corresponding author.
 Email: \texttt{zxm23@mails.tsinghua.edu.cn}.}\\[1ex]
 \small Department of Mathematical Sciences, Tsinghua University,
 \small Beijing 100084, China
}
\date{}

\begin{document}
\maketitle

\begin{abstract}
A family of subsets of $[n]$ is $r$-fork-free if none of its members is
strictly contained in $r$ other distinct members. For each fixed integer
$r\ge2$, we prove that the maximum size of such a family is
\[
 \binom{n}{\lfloor n/2\rfloor}
 \left(1+\frac{2(r-1)}{n}+o(n^{-1})\right).
\]
This determines the second-order term and matches the upper bound of
De Bonis and Katona. Our lower bound comes from a construction on two
adjacent middle levels, using a finite ordered collection of disjoint
coordinate blocks and a condition on sums modulo $n$. For each prescribed error,
the blocks are fixed before $n$ tends to infinity, so the construction
works for every sufficiently large $n$.
\end{abstract}

\section{Introduction}

Let $[n]=\{1,2,\ldots,n\}$ and
$\binom{[n]}k=\{S\subseteq[n]:|S|=k\}$. An antichain in $2^{[n]}$, the
power set of $[n]$, is a family with no two distinct members comparable
by inclusion. Sperner proved the following sharp bound.

\begin{theorem}[Sperner~\cite{Sperner1928}]\label{thm:sperner}
If $\mathcal F\subseteq2^{[n]}$ is an antichain, then
\[
 |\mathcal F|\le\binom{n}{\lfloor n/2\rfloor}.
\]
Equality is attained by the family of all $\lfloor n/2\rfloor$-element
subsets of $[n]$.
\end{theorem}

Since then, the sizes and constructions of set families under various
restrictions have attracted much attention. Let $r\ge2$ be an integer.
A family $\mathcal F\subseteq2^{[n]}$ is called {\em $r$-fork-free} if there are
no distinct members $F,F_1,\ldots,F_r\in\mathcal F$ such that
$F\subseteq F_i$ for every $1\le i\le r$. Put
\[
 \La(n,V_r)=\max\{|\mathcal F|:\mathcal F\subseteq2^{[n]}
                  \text{ is }r\text{-fork-free}\}.
\]
We also write $M_n=\binom{n}{\lfloor n/2\rfloor}$. Throughout the paper,
$r$ is fixed in every asymptotic statement. Subscripts on $O(\cdot)$
indicate the parameters on which the implicit constant may depend.

Katona and Tarj\'an considered the $r=2$ case
and obtained the following bounds.

\begin{theorem}[Katona and Tarj\'an~\cite{KatonaTarjan1983}]\label{thm:two-fork}
As $n\to\infty$,
\[
 M_n\left(1+\frac1n+O(n^{-2})\right)
 \le\La(n,V_2)\le M_n\left(1+\frac2n\right).
\]
\end{theorem}

Thanh considered the $r\ge2$ case. De Bonis and
Katona later improved the second-order term in
Thanh's upper bound, giving the following result.

\begin{theorem}[Thanh~\cite{Thanh1998}; De Bonis and Katona~\cite{DeBonisKatona2007}]\label{thm:known}
For every fixed integer $r\ge2$, as $n\to\infty$,
\[
 M_n\left(1+\frac{r-1}{n}+O_r(n^{-2})\right)
 \le\La(n,V_r)
 \le M_n\left(1+\frac{2(r-1)}n+O_r(n^{-2})\right).
\]
\end{theorem}

There is a gap between the lower and upper bounds in these two theorems.
We are interested in constructing large $r$-fork-free families. We first
briefly recall the previous constructions.

Graham and Sloane~\cite{GrahamSloane1980} used the following construction.
Let $G$ be a finite abelian group of order $n$. For $a\in G$ and
$1\le h\le n$, define
\[
 \mathcal G_a(h)=
 \left\{A\in\binom Gh:\sum_{x\in A}x=a\right\}.
\]
Then $\sum_{a\in G}|\mathcal G_a(h)|=\binom nh$. So there is an
$a_0\in G$ such that
\[
 |\mathcal G_{a_0}(h)|
 =\max_{a\in G}|\mathcal G_a(h)|\ge\frac1n\binom nh.
\]
For an $(h-1)$-set $A\subseteq G$, an extension $A\cup\{x\}$ belongs to
$\mathcal G_{a_0}(h)$ only if
$x=a_0-\sum_{y\in A}y$. Thus there is at most one such extension.
Taking $h=\lfloor n/2\rfloor+1$, the family
$\binom G{h-1}\cup\mathcal G_{a_0}(h)$ is $2$-fork-free and gives the
lower bound in Theorem~\ref{thm:two-fork}.

Thanh~\cite{Thanh1998} extended this construction. Put
$q=\lceil n/(r-1)\rceil$. For $a\in\{0,\ldots,q-1\}$ and $1\le h\le n$,
define
\[
 \mathcal G_r(a,h)=
 \left\{A\in\binom{[n]}h:\sum_{x\in A}x\equiv a\pmod q\right\}.
\]
Then $\sum_{a=0}^{q-1}|\mathcal G_r(a,h)|=\binom nh$. So there is an
$a_0\in\{0,\ldots,q-1\}$ such that
\[
 |\mathcal G_r(a_0,h)|\ge\frac1q\binom nh.
\]
For each $(h-1)$-set $A$, all elements $x\in[n]\setminus A$ for which
$A\cup\{x\}\in\mathcal G_r(a_0,h)$ lie in one residue class modulo $q$.
Every such class in $[n]$ has at most $\lceil n/q\rceil\le r-1$ elements.
Hence $\binom{[n]}{h-1}\cup\mathcal G_r(a_0,h)$ is $r$-fork-free.
Taking $h=\lfloor n/2\rfloor+1$ gives the lower bound in
Theorem~\ref{thm:known}.

In particular, Theorem~\ref{thm:known} gives
\[
 1\le\liminf_{n\to\infty}
 \frac{n(\La(n,V_r)-M_n)}{(r-1)M_n}
 \le\limsup_{n\to\infty}
 \frac{n(\La(n,V_r)-M_n)}{(r-1)M_n}\le2.
\]
These bounds do not determine whether the limit exists. Our main result
determines the limit and the second-order term.

\begin{theorem}\label{thm:main}
For every fixed integer $r\ge2$,
\[
 \La(n,V_r)=M_n\left(1+\frac{2(r-1)}n+o(n^{-1})\right)
\]
as $n\to\infty$.
\end{theorem}

This paper is organized as follows. Section~\ref{sec:construction} gives
the construction. Section~\ref{sec:parameters} defines the parameters
used to count its members and proves their recurrence relations.
Section~\ref{sec:growth} proves Theorem~\ref{thm:main} by a finite
sequence of extensions of the tuple.

\section{Construction}\label{sec:construction}

Let $r\ge2$, $n\ge2$ and $\ell\ge1$ be integers. Put
\[
 h=\left\lfloor\frac n2\right\rfloor+1,\qquad
 \tau=n-2(h-1)\in\{0,1\}.
\]
Fix an ordered tuple $\bm=(m_1,\ldots,m_\ell)$ of integers with
$m_i\ge r$ for $1\le i\le\ell$, and let $d=\sum_{i=1}^{\ell}m_i$.
Assume that $n\ge\max\{2,d,r-1\}$. Choose pairwise disjoint blocks
$X_1,\ldots,X_\ell\subseteq[n]$ with $|X_i|=m_i$, and put
$X=\bigcup_{i=1}^{\ell}X_i$. Then $|X|=d$. We keep this ordered partition of $X$
fixed throughout the construction. For any set $P\subseteq [n]$, we call $P\cap X$ the $trace~of~ P~ on~ X$ or simply $trace$. Similarly, for $\mathcal{P}\subseteq 2^{[n]}$, we call $\{P\cap X:P\in\mathcal{P}\}$ the $traces~of~\mathcal{P}~ on~ X$ or simply $traces$.

For $P\subseteq X$ and $1\le i\le\ell$, define the number of missing
elements in the $i$-th block by
\[
 \delta_i(P)=|X_i\setminus P|=m_i-|P\cap X_i|.
\]
Using $\delta_i(P)=r-1$ as the boundary, define
\[
 \cU_i=\{P\subseteq X:\delta_j(P)\ge r\text{ for }1\le j<i,
                         \ \delta_i(P)\le r-2\},
\]
and put
\begin{equation}\label{eq:trace-families}
 \cU=\bigcup_{i=1}^{\ell}\cU_i,\qquad
 \cV=\{P\subseteq X:\delta_i(P)\ge r
                        \text{ for }1\le i\le\ell\}.
\end{equation}
There is no restriction on the blocks after $X_i$ in the definition
of $\cU_i$. The families $\cU_1,\ldots,\cU_\ell,\cV$ are pairwise
disjoint: if $i<j$, a member of $\cU_i$ has $\delta_i(P)\le r-2$,
whereas a member of $\cU_j$ or $\cV$ has $\delta_i(P)\ge r$.

Fix a bijection $\iota:[n]\to\mathbb Z_n$, where $\mathbb Z_n$ is the
additive group of integers modulo $n$. Define
$\sigma:2^{[n]}\to\mathbb Z_n$ by
\[
 \sigma(A)=\sum_{x\in A}\iota(x)\qquad(A\subseteq[n]).
\]
In particular, $\sigma(\emptyset)=0$. Select a fixed set
$C\subseteq\mathbb Z_n$ with $|C|=r-1$. For the fixed tuple $\bm$
and blocks above, define
\begin{align}
 \mathcal A&=\left\{A\in\binom{[n]}{h-1}:A\cap X\notin\cU\right\},
                                                        \notag\\
 \mathcal B&=\left\{A\in\binom{[n]}h:A\cap X\in\cU\right\},
                                                        \notag\\
 \mathcal C_a&=\left\{A\in\binom{[n]}h:A\cap X\in\cV,\
                          \sigma(A)\in a+C\right\}
                          \quad(a\in\mathbb Z_n),          \label{eq:parts}
\end{align}
where $a+C=\{a+c:c\in C\}$. Our construction is
\begin{equation}\label{eq:family}
 \mathcal F_a(\bm,X)=\mathcal A\cup\mathcal B\cup\mathcal C_a.
\end{equation}
These three parts are pairwise disjoint. Indeed, $\mathcal A$ lies
in the $(h-1)$-th level, while $\mathcal B$ and $\mathcal C_a$ lie in
the $h$-th level and have traces in the disjoint families $\cU$ and
$\cV$.

\begin{lemma}\label{lem:free}
For every tuple $\bm$ and choice of blocks satisfying the conditions
above, and every $a\in\mathbb Z_n$, the family $\mathcal F_a(\bm,X)$
is $r$-fork-free.
\end{lemma}

\begin{proof}
By~\eqref{eq:family}, every member of $\mathcal F_a(\bm,X)$ has size
$h-1$ or $h$. A member of $\mathcal B\cup\mathcal C_a$ has no strict
superset in this family. Thus only a member of $\mathcal A$ could
be the lower set of an $r$-fork. Fix $A\in\mathcal A$. Then $A\cap X\notin\cU$.
We count the
elements $x\in[n]\setminus A$ for which
$A\cup\{x\}\in\mathcal F_a(\bm,X)$, and prove that there are at most
$r-1$ such elements. There are two cases.

\smallskip
\Needspace{4\baselineskip}
\noindent\textbf{Case 1.} $A\cap X\notin\cV$.

In this case, there is $1\le i\le \ell$ such that $\delta_i(A\cap X)\le r-1$.
Let $i$ be the smallest index such that $\delta_i(A\cap X)\le r-1$.
Then $\delta_j(A\cap X)\ge r$ for every $1\le j<i$. Since
$A\cap X\notin\cU$, we have $\delta_i(A\cap X)=r-1$.
For each $x\in X_i\setminus A$, we have
\[
 \delta_j((A\cup\{x\})\cap X)=\delta_j(A\cap X)\ge r
       \quad(1\le j<i),
 \qquad
 \delta_i((A\cup\{x\})\cap X)=r-2.
\]
Thus $(A\cup\{x\})\cap X\in\cU_i$ and then $A\cup\{x\}\in\mathcal B$.
There are exactly $|X_i\setminus A|=r-1$ such $x$.

Now let $x\in[n]\setminus(A\cup X_i)$ and we consider $A\cup \{x\}$. If $x\notin X$, the trace of $A\cup\{x\}$
does not change and is outside $\cU\cup\cV$. If there is $j<i$ such that $x\in X_j$, then $\delta_j((A\cup\{x\})\cap X)\ge r-1$. Since $j<i$, by the choice of $i$, we have $\delta_j((A\cup\{x\})\cap X)\ge r$. So $(A\cup\{x\})\cap X\notin \cU\cup\cV$. If there is $j>i$ such that $x\in X_j$, $\delta_i((A\cup\{x\})\cap X)=r-1$ and $(A\cup\{x\})\cap X\notin \cU\cup\cV$.
Therefore none of these extensions belongs to
$\mathcal F_a(\boldsymbol m,X)$.

Thus exactly $r-1$ extensions are selected in this case.

\smallskip
\Needspace{4\baselineskip}
\noindent\textbf{Case 2.} $A\cap X\in\cV$.

In this case,
for every $1\le i\le\ell$, we have $\delta_i(A\cap X)\ge r$. Then for every $x\in[n]\setminus A$, we have $\delta_i((A\cup\{x\})\cap X)\ge r-1$ which implies
$(A\cup\{x\})\cap X\notin\cU$.
Hence an extension can belong to $\mathcal F_a(\bm,X)$ only if it
belongs to $\mathcal C_a$. This requires
\[
 \sigma(A)+\iota(x)=\sigma(A\cup\{x\})\in a+C,
\]
or equivalently $\iota(x)\in a-\sigma(A)+C$. This last set has
$r-1$ elements, and $\iota$ is injective. Therefore at most $r-1$
extensions are selected in this case as well.
\end{proof}

\section{Counting parameters}\label{sec:parameters}

Fix a tuple $\bm=(m_1,\ldots,m_\ell)$ and blocks
$X_1,\ldots,X_\ell$ as in Section~\ref{sec:construction}.
Recall the definitions of $X$, $d=\sum_{i=1}^{\ell}m_i$,
$h=\lfloor n/2\rfloor+1$, $\tau$, $\cU$ and $\cV$.
For $P\subseteq X$, put
\[
 L_P=\binom{n-d}{h-1-|P|},\qquad
 U_P=\binom{n-d}{h-|P|}.
\]
Here $L_P$ and $U_P$ count, respectively, the $(h-1)$-sets and
$h$-sets whose trace on $X$ equals $P$. We take $\binom{N}{k}=0$ when $0\le k\le N$ does not hold.

\begin{lemma}\label{lem:average}
For every fixed tuple $\bm$ and blocks as above, there is an
$a\in\mathbb Z_n$ such that
\begin{equation}\label{eq:average}
 |\mathcal F_a(\bm,X)|\ge
 M_n+\sum_{P\in\cU}(U_P-L_P)+\frac{r-1}{n}\sum_{P\in\cV}U_P.
\end{equation}
\end{lemma}

\begin{proof}
Recall the three families in~\eqref{eq:parts}. The family $\mathcal A$
contains all $(h-1)$-sets except those with trace in $\cU$. For each
$P\in\cU$, exactly $L_P$ such sets have trace $P$, and distinct
traces give disjoint classes. Therefore
\[
 |\mathcal A|=M_n-\sum_{P\in\cU}L_P.
\]
Similarly, $\mathcal B$ consists of the $h$-sets with trace in $\cU$,
so
\[
 |\mathcal B|=\sum_{P\in\cU}U_P.
\]
Fix an $h$-set $A$ with $A\cap X\in\cV$. For each $c\in C$, the
unique translate for which $\sigma(A)=a+c$ is $a=\sigma(A)-c$.
These $r-1$ values of $a$ are distinct. Thus $A$ is counted in
exactly $r-1$ of the families $\mathcal C_a$. Hence
\[
 \sum_{a\in\mathbb Z_n}|\mathcal C_a|
 =(r-1)\sum_{P\in\cV}U_P.
\]
There is consequently an $a_0\in\mathbb Z_n$ for which
$|\mathcal C_{a_0}|\ge (r-1)n^{-1}\sum_{P\in\cV}U_P$.
Adding the three counts in~\eqref{eq:family} proves~\eqref{eq:average}.
\end{proof}

Recall that $\tau=n-2(h-1)\in\{0,1\}$. We use the following
asymptotic estimates.

\begin{lemma}\label{lem:traces}
For fixed $d$, as $n\to\infty$, the estimates
\begin{align}
 L_P&=M_n\bigl(2^{-d}+O_d(n^{-1})\bigr), \notag\\
 U_P&=M_n\bigl(2^{-d}+O_d(n^{-1})\bigr), \notag\\
 n(U_P-L_P)&=M_n\bigl(2^{1-d}(\tau-d-1+2|P|)
                                      +O_d(n^{-1})\bigr)       \label{eq:trace-estimates}
\end{align}
hold uniformly over $P\subseteq X$ and both values of $\tau$.
\end{lemma}

\begin{proof}Given $P\subseteq X$,
put $w=|P|$. For a nonnegative integer $s$, write
$(z)_s=z(z-1)\cdots(z-s+1)$, with $(z)_0=1$.
For all sufficiently large $n$, expanding the binomial coefficients gives
\[
 \frac{L_P}{M_n}
 =\frac{(h-1)_w(n-h+1)_{d-w}}{(n)_d}
 =\frac{n^d(2^{-d}+O_d(n^{-1}))}
        {n^d(1+O_d(n^{-1}))}
 =2^{-d}+O_d(n^{-1}).
\]
The middle equality follows because $h-1=(n-\tau)/2$ and there
are only $d$ factors in each product. This also shows uniformity
for $0\le w\le d$ and $\tau\in\{0,1\}$.
The ratio of the two binomial coefficients gives
\[
 \frac{U_P}{L_P}
 =\frac{n-d-h+w+1}{h-w}
 =\frac{h+\tau-1+w-d}{h-w}
 =1+O_d(n^{-1}),
\]
which proves the estimate for $U_P$. Finally,
\begin{align*}
 U_P-L_P
 &=L_P\frac{\tau-d-1+2w}{h-w}\\
 &=\frac{M_n}{n}\bigl(2^{-d}+O_d(n^{-1})\bigr)
            \frac{n}{h-w}(\tau-d-1+2w)\\
 &=\frac{M_n}{n}
       \bigl(2^{1-d}(\tau-d-1+2w)+O_d(n^{-1})\bigr),
\end{align*}
where the last equality uses $n/(h-w)=2+O_d(n^{-1})$.
\end{proof}

For the fixed tuple $\bm$, define
\begin{align}
 u=u(\bm)&=\sum_{P\in\cU}2^{-d}=2^{-d}|\cU|,
 &v=v(\bm)&=\sum_{P\in\cV}2^{-d}=2^{-d}|\cV|, \notag\\
 \alpha=\alpha(\bm)&=\sum_{P\in\cU}(2|P|-d)2^{-d},
 &\beta=\beta(\bm)&=\frac1{|\cV|}
                         \sum_{P\in\cV}(2|P|-d).              \label{eq:parameters}
\end{align}
The empty set belongs to $\cV$ because $m_i\ge r$ for every $i$.
Hence $v>0$ and $\beta$ is well-defined. These parameters depend only
on $\bm$: bijections between corresponding blocks preserve the
trace families and the sizes of all traces.

\begin{proposition}\label{prop:coefficient}
For every fixed tuple $\bm$, there is an $a\in\mathbb Z_n$ such that
\begin{equation}\label{eq:fixed-count}
 |\mathcal F_a(\bm,X)|\ge
 M_n\left(1+\frac{c_{r,\tau}(\bm)+O_{r,\bm}(n^{-1})}{n}\right),
\end{equation}
where
\begin{equation}\label{eq:coefficient}
 c_{r,\tau}(\bm)=2\alpha+2(\tau-1)u+(r-1)v.
\end{equation}
\end{proposition}

\begin{proof}
Apply~\eqref{eq:trace-estimates} to~\eqref{eq:average}. Since $\bm$
is fixed, $d$, $|\cU|$ and $|\cV|$ are fixed. Thus the sum of the
error terms is still $O_{r,\bm}(n^{-1})$, and for some $a$,
\[
 |\mathcal F_a(\bm,X)|\ge
 M_n\left(1+\frac{
  \sum_{P\in\cU}2^{1-d}(\tau-d-1+2|P|)
  +(r-1)\sum_{P\in\cV}2^{-d}
  +O_{r,\bm}(n^{-1})}{n}\right).
\]
By~\eqref{eq:parameters},
\[
 \sum_{P\in\cU}2^{1-d}(\tau-d-1+2|P|)
 =2\alpha+2(\tau-1)u,\qquad
 (r-1)\sum_{P\in\cV}2^{-d}=(r-1)v.
\]
Substitution gives~\eqref{eq:fixed-count}.
\end{proof}

We next compare a tuple $\bm=(m_1,\ldots,m_\ell)$ with $\ell\ge1$,
and its extension $\bm'=(m_1,\ldots,m_\ell,m)$, where $m\ge r$.
Write $u',v',\alpha',\beta'$ for the parameters of $\bm'$.
Choose a new block $X_{\ell+1}$ of size $m$ disjoint from $X$ and
put $X'=X\cup X_{\ell+1}$. Define $\cU',\cV'$ from these blocks as
in~\eqref{eq:trace-families}. At this stage we work only with
finite block sets; there is no need to fix an value of $n$.
Put
\[
 \begin{aligned}
 D_m&=\sum_{j=0}^{r-1}\binom mj,&
 G_m&=\sum_{j=0}^{r-1}(m-2j)\binom mj,\\
 H_m&=\sum_{j=0}^{r-2}\binom mj,&
 E_m&=\sum_{j=0}^{r-2}(m-2j)\binom mj.
 \end{aligned}
\]

\begin{lemma}\label{lem:recurrence}
Let $\bm=(m_1,\ldots,m_\ell)$ with $\ell\ge1$ and $m_i\ge r$,
and let $\bm'=(m_1,\ldots,m_\ell,m)$ with $m\ge r$.
Their parameters satisfy
\begin{equation}\label{eq:recurrences}
 \begin{aligned}
 v'&=v(1-2^{-m}D_m),&
 u'&=u+v\,2^{-m}H_m,\\
 \alpha'&=\alpha+v\,2^{-m}(\beta H_m+E_m),&
 \beta'&=\beta-\frac{G_m}{2^m-D_m}.
 \end{aligned}
\end{equation}
Moreover, for $\tau\in\{0,1\}$,
\begin{equation}\label{eq:coefficient-increment}
 c_{r,\tau}(\bm')-c_{r,\tau}(\bm)
 =(v-v')\rho_{r,\tau}(m,\beta),
\end{equation}
where
\begin{equation}\label{eq:rho}
 \rho_{r,\tau}(m,\beta)
 =(r-1)+2(\beta+\tau-r)\frac{H_m}{D_m}.
\end{equation}
\end{lemma}

\begin{proof}
Every $P'\subseteq X'$ has a unique representation $P'=P\cup A$,
where $P\subseteq X$ and $A\subseteq X_{\ell+1}$. By definition,
$P'\in\cV'$ if and only if $P\in\cV$ and $|A|\le m-r$. Hence
\begin{equation}\label{eq:V-prime}
 |\cV'|=|\cV|\sum_{j=0}^{m-r}\binom mj
        =|\cV|(2^m-D_m).
\end{equation}
Using~\eqref{eq:parameters}, we obtain
\[
 v'=2^{-d-m}|\cV'|=v(1-2^{-m}D_m).
\]

\Needspace{5\baselineskip}
The member $P'$ of $\cU'$ falls into two disjoint classes:
\begin{enumerate}
 \item $P'=P\cup A$, where $P\in\cU$ and $A\subseteq X_{\ell+1}$;
 \item $P'=P\cup A$, where $P\in\cV$, $A\subseteq X_{\ell+1}$
       and $|A|\ge m-r+2$.
\end{enumerate}
The first class contains $2^m|\cU|$ sets. The second class  contains
$H_m|\cV|$ sets, because
$\sum_{j=m-r+2}^{m}\binom mj=H_m$. Therefore
\[
 u'=2^{-d-m}|\cU'|=u+v\,2^{-m}H_m.
\]

For the parameter $\alpha'$, the contribution of the first class is
\begin{align*}
 &2^{-d-m}\sum_{P\in\cU}\sum_{A\subseteq X_{\ell+1}}
                         (2|P|+2|A|-d-m)\\
 &\hspace{2em}=2^{-d}\sum_{P\in\cU}(2|P|-d)=\alpha.
\end{align*}
Here $\sum_{A\subseteq X_{\ell+1}}(2|A|-m)=0$, by pairing each
$A$ with its complement in $X_{\ell+1}$. Also,
\[
 \sum_{\substack{A\subseteq X_{\ell+1}\\|A|\ge m-r+2}}
             (2|A|-m)=E_m,
\]
as follows by writing $j=m-|A|$. Thus the second class contributes
\begin{align*}
 &2^{-d-m}\sum_{P\in\cV}
     \sum_{\substack{A\subseteq X_{\ell+1}\\|A|\ge m-r+2}}
                         (2|P|+2|A|-d-m)\\
 &\hspace{2em}=2^{-d-m}
     \left(H_m\sum_{P\in\cV}(2|P|-d)+|\cV|E_m\right)\\
 &\hspace{2em}=v\,2^{-m}(\beta H_m+E_m).
\end{align*}
Adding these contributions proves the recurrence for $\alpha'$.

Next, using~\eqref{eq:V-prime}, we have
\begin{align*}
 \beta'
 &=\frac1{|\cV'|}
      \sum_{P\in\cV}
      \sum_{\substack{A\subseteq X_{\ell+1}\\|A|\le m-r}}
                         (2|P|+2|A|-d-m)\\
 &=\beta+\frac1{2^m-D_m}
          \sum_{\substack{A\subseteq X_{\ell+1}\\|A|\le m-r}}
                         (2|A|-m)\\
 &=\beta-\frac{G_m}{2^m-D_m}.
\end{align*}
For the last equality, the sum over all $A\subseteq X_{\ell+1}$ is
zero, while the sum over $|A|\ge m-r+1$, on putting $j=m-|A|$,
equals $G_m$. The denominator is positive because $m\ge r$.

Substituting the four recurrences into~\eqref{eq:coefficient}, and
using $v-v'=v\,2^{-m}D_m$, gives~\eqref{eq:coefficient-increment}
with
\[
 \rho_{r,\tau}(m,\beta)
 =2\frac{E_m}{D_m}+2(\beta+\tau-1)\frac{H_m}{D_m}-(r-1).
\]
To simplify this expression, use the identity
\[
 (m-2j)\binom mj
 =m\left(\binom{m-1}j-\binom{m-1}{j-1}\right).
\]
Summing up to $r-2$ and up to $r-1$, respectively, yields
\begin{equation}\label{eq:binomial-identities}
 E_m=(r-1)\binom m{r-1},\qquad
 G_m=(m-r+1)\binom m{r-1}.
\end{equation}
Since $D_m=H_m+\binom m{r-1}$, it follows that
\begin{align*}
 \rho_{r,\tau}(m,\beta)
 &=2(r-1)\frac{D_m-H_m}{D_m}
       +2(\beta+\tau-1)\frac{H_m}{D_m}-(r-1)\\
 &=(r-1)+2(\beta+\tau-r)\frac{H_m}{D_m},
\end{align*}
which is~\eqref{eq:rho}.
\end{proof}

We will use the following estimates.

\begin{proposition}\label{prop:estimates}
For fixed $r\ge2$, as $m\to\infty$,
\begin{equation}\label{eq:block-asymptotics}
 \frac{H_m}{D_m}=\frac{r-1}{m}+O_r(m^{-2}),\qquad
 2^{-m}D_m\longrightarrow0.
\end{equation}
For the extension in Lemma~\ref{lem:recurrence}, if $2^{-m}D_m\le1/2$,
then
\begin{equation}\label{eq:beta-drift}
 0<\beta-\beta'\le2m\,2^{-m}D_m.
\end{equation}
\end{proposition}

\begin{proof}
The ratio
\[
 \frac{\binom m{r-2}}{\binom m{r-1}}
 =\frac{r-1}{m-r+2}
 =\frac{r-1}{m}+O_r(m^{-2})
\]
gives the leading term of $H_m/\binom m{r-1}$. The sum of the terms
$\binom mj/\binom m{r-1}$ with $0\le j\le r-3$ is
$O_r(m^{-2})$; for $r=2$ this sum is empty. Therefore
\[
 \frac{H_m}{D_m}
 =\frac{\binom m{r-1}\bigl((r-1)/m+O_r(m^{-2})\bigr)}
        {\binom m{r-1}\bigl(1+O_r(m^{-1})\bigr)}
 =\frac{r-1}{m}+O_r(m^{-2}).
\]
Also $D_m\le r m^{r-1}$, so $2^{-m}D_m\to0$.

By~\eqref{eq:binomial-identities}, $0<G_m\le mD_m$ for every
$m\ge r$. If $2^{-m}D_m\le1/2$, \eqref{eq:recurrences}
gives
\[
 0<\beta-\beta'
 =\frac{2^{-m}G_m}{1-2^{-m}D_m}
 \le2m\,2^{-m}D_m.\qedhere
\]
\end{proof}

For convenience, we now include the empty tuple $\bm=\emptyset$,
corresponding to $\ell=0$, by putting
\begin{equation}\label{eq:empty}
 v=1,\qquad u=\alpha=\beta=0,\qquad
 c_{r,\tau}(\emptyset)=r-1.
\end{equation}
These values agree with the recurrences. Indeed, for a one-entry
tuple $(m)$, direct counting gives
\[
 v=1-2^{-m}D_m,\qquad u=2^{-m}H_m,\qquad
 \alpha=2^{-m}E_m,\qquad
 \beta=-\frac{G_m}{2^m-D_m}.
\]
These are exactly the values obtained by applying~\eqref{eq:recurrences}
to~\eqref{eq:empty}. Hence Lemma~\ref{lem:recurrence} and
the drift bound~\eqref{eq:beta-drift} also apply to an extension
of the empty tuple.

\section{Finite growth}\label{sec:growth}

We first prove that a finite sequence of extensions gives a tuple
whose coefficient is arbitrarily close to $2(r-1)$.

\begin{lemma}\label{lem:growth}
For every fixed integer $r\ge2$ and every $\varepsilon>0$, there is
a finite tuple $\bm$, depending only on $r$ and $\varepsilon$,
such that
\[
 c_{r,\tau}(\bm)>2(r-1)-\varepsilon
 \qquad\text{for both }\tau\in\{0,1\}.
\]
\end{lemma}

\begin{proof}
It is enough to consider $0<\varepsilon<r-1$. Choose
\[
 0<\eta<\min\left\{\frac1{32},
                   \frac{\varepsilon}{17(r-1)}\right\},
 \qquad
 T=\left\lceil\frac2\eta\log\frac1\eta\right\rceil.
\]
We construct $\bm^{(T)}$ in $T$ steps.

\smallskip
\noindent\textbf{Step 0.}
Put $\bm^{(0)}=\emptyset$. By~\eqref{eq:empty},
$c_{r,\tau}(\bm^{(0)})=r-1$, $v^{(0)}=1$ and $\beta^{(0)}=0$.

\smallskip
\noindent\textbf{Step $i$, where $1\le i\le T$.}
Suppose that the preceding step gives
$\bm^{(i-1)}=(m_1,\ldots,m_\ell)$ for some $\ell\ge0$.
Its parameter $\beta^{(i-1)}$ is finite. By
\eqref{eq:block-asymptotics} and $\beta^{(i-1)}$ being finite, we can choose an integer $m\ge r$
so large that
\begin{equation}\label{eq:choice}
 |\beta^{(i-1)}|\le\eta m,\qquad
 2^{-m}D_m\le\frac\eta2,\qquad
 \frac{H_m}{D_m}\le\frac{2(r-1)}m,\qquad
 \frac rm\le\eta.
\end{equation}
Let $s=\left\lfloor\eta/(2^{-m}D_m)\right\rfloor$ and append $s$
entries equal to $m$ to obtain
\[
 \bm^{(i)}=(m_1,\ldots,m_\ell,m_{\ell+1},\ldots,m_{\ell+s}),
 \qquad m_{\ell+1}=\cdots=m_{\ell+s}=m.
\]
For convenience, put
$\bm_j=(m_1,\ldots,m_\ell,m_{\ell+1},\ldots,m_{\ell+j})$
for $0\le j\le s$, and write $u_j,v_j,\alpha_j,\beta_j$
for its parameters. In particular,
$\bm_0=\bm^{(i-1)}$ and $\bm_s=\bm^{(i)}$.
The definition of $s$ and~\eqref{eq:choice} give
\[
 \frac\eta2\le \eta-2^{-m}D_m\le s\,2^{-m}D_m\le\eta.
\]
Since $\eta<1/32$, the drift bound~\eqref{eq:beta-drift} applies
at every extension. For $0\le j\le s$, it gives
\[
 |\beta_j|
 \le|\beta^{(i-1)}|+\sum_{t=0}^{j-1}|\beta_{t+1}-\beta_t|
 \le\eta m+2jm\,2^{-m}D_m
 \le3\eta m.
\]
By~\eqref{eq:rho} and~\eqref{eq:choice}, for $0\le j<s$ and
$\tau\in\{0,1\}$,
\begin{align*}
 \rho_{r,\tau}(m,\beta_j)
 &=(r-1)+2(\beta_j+\tau-r)\frac{H_m}{D_m}\\
 &\ge(r-1)+2(\beta_j-r)\frac{H_m}{D_m}\\
 &\ge(r-1)-2(r+3\eta m)\frac{2(r-1)}m\\
 &\ge(r-1)-16(r-1)\eta>0.
\end{align*}
Thus~\eqref{eq:coefficient-increment} gives
\begin{align}
 c_{r,\tau}(\bm^{(i)})-c_{r,\tau}(\bm^{(i-1)})
 &=\sum_{t=0}^{s-1}
        \bigl(c_{r,\tau}(\bm_{t+1})-c_{r,\tau}(\bm_t)\bigr)
                                                        \notag\\
 &\ge(v^{(i-1)}-v^{(i)})
                \bigl((r-1)-16(r-1)\eta\bigr).             \label{eq:step-gain}
\end{align}
Here we used $\sum_{t=0}^{s-1}(v_t-v_{t+1})=v^{(i-1)}-v^{(i)}$.
Also, by~\eqref{eq:recurrences},
\[
 v^{(i)}=v_s=v^{(i-1)}(1-2^{-m}D_m)^s
 \le v^{(i-1)}e^{-s2^{-m}D_m}
 \le v^{(i-1)}e^{-\eta/2}.
\]

Repeat this step $T$ times. At each step the current tuple and its
parameters are finite, so the choice in~\eqref{eq:choice} is possible.
After the last step, we have
\[
 v^{(T)}\le v^{(0)}e^{-T\eta/2}\le\eta.
\]
Summing~\eqref{eq:step-gain} over $1\le i\le T$ gives
\[
 c_{r,\tau}(\bm^{(T)})-c_{r,\tau}(\bm^{(0)})
 \ge(1-v^{(T)})\bigl((r-1)-16(r-1)\eta\bigr).
\]
Since the last factor is positive, it follows that
\begin{align*}
 c_{r,\tau}(\bm^{(T)})
 &\ge(r-1)+(1-\eta)\bigl((r-1)-16(r-1)\eta\bigr)\\
 &\ge2(r-1)-17(r-1)\eta
 >2(r-1)-\varepsilon.
\end{align*}
There are only $T$ steps, and both $m$ and $s$ are finite at every
step. Thus the final tuple is finite and depends only on $r$ and
$\varepsilon$. All choices and estimates hold for both values of
$\tau$, so the same tuple works for both parities.
\end{proof}

\begin{proof}[Proof of Theorem~\ref{thm:main}]
Fix $r\ge2$ and $\varepsilon>0$.
By Lemma~\ref{lem:growth}, there is a finite tuple $\bm$ such that
\[
 c_{r,\tau}(\bm)>2(r-1)-\frac\varepsilon2
 \qquad(\tau\in\{0,1\}).
\]
Fix this tuple and let $d$ be the sum of its entries.
For $n\ge N_1=\max\{2,d,r-1\}$, choose the corresponding blocks
inside $[n]$. Proposition~\ref{prop:coefficient} supplies an
$a\in\mathbb Z_n$ such that
\[
 |\mathcal F_a(\bm,X)|\ge
 M_n\left(1+\frac{2(r-1)-\varepsilon/2
                         +O_{r,\bm}(n^{-1})}{n}\right).
\]
Since the tuple was fixed before $n$ tends to infinity, there is
an $N_2$ such that the error in the numerator is at least
$-\varepsilon/2$ for every $n\ge N_2$. With
$n_0=\max\{N_1,N_2\}$, we obtain
\[
 |\mathcal F_a(\bm,X)|\ge
 M_n\left(1+\frac{2(r-1)-\varepsilon}{n}\right)
 \qquad(n\ge n_0).
\]
By Lemma~\ref{lem:free}, this family is $r$-fork-free.
Hence
\[
 \liminf_{n\to\infty}
 n\left(\frac{\La(n,V_r)}{M_n}-1\right)\ge2(r-1).
\]
The upper bound in Theorem~\ref{thm:known} gives the reverse
inequality for the limit superior. This proves the theorem.
\end{proof}
\section*{Acknowledgement}
This research is supported by the National Natural Science Foundation of China  (Grant 12571372).

\section*{Declaration of competing interest}
The authors declare that they have no known competing financial interests or personal relationships that could have appeared to influence the work reported in this paper.

\section*{Data availability}
No data was used for the research described in the article.

\section*{Declaration of AI use}

The authors used AI tools for preliminary research assistance and
auxiliary computations. All proofs and calculations have been checked
by the authors, who take full responsibility for their correctness.

\end{document}